\documentclass[english,reqno]{amsart}

\usepackage{latexsym}
\usepackage{amssymb,amsbsy,amsmath,amsfonts,amssymb,amscd}
\usepackage{mathrsfs}

\usepackage{float}
\usepackage[caption = false]{subfig}
\usepackage[final]{graphicx}

\usepackage{subfig}
\usepackage{graphicx}

\usepackage{color}

\usepackage{cases}
\theoremstyle{plain}

\newtheorem{theorem}{Theorem}
 
\newtheorem{lemma}{Lemma}
\newtheorem{proposition}{Proposition}

\theoremstyle{definition} 

\newtheorem{remark}{Remark}

\renewcommand{\vec}[1]{\mathbf{#1}}
\newcommand{\parallelsum}{\mathbin{\!/\mkern-5mu/\!}}
\newcommand{\eps}{\varepsilon}
\newcommand{\average}[1]{ \langle#1 \rangle}

\newcommand{\Amat}{\mathsf{A}}

\newcommand{\Bmat}{\mathsf{B}}
\newcommand{\Rmat}{\mathsf{R}}

\newcommand{\Emat}{\mathsf{E}}

\newcommand{\Xmat}{\mathsf{X}}

\newcommand{\Vmat}{\mathsf{V}}
\newcommand{\Jmat}{\mathsf{J}}
\newcommand{\Imat}{\mathsf{I}}

\newcommand{\vecf}{\vec{f}}

\newcommand{\opA}{\mathcal{A}}
\newcommand{\opB}{\mathcal{B}}

\newcommand{\opJ}{\mathcal{J}}

\newcommand{\rd}{\mathrm{d}}

\newcommand{\RR}{\mathbb{R}}
\newcommand{\Sp}{\mathbb{S}}
\newcommand{\Kn}{\mathsf{Kn}}

\newcommand{\fvec}{\vec{f}}

\title{A new numerical approach to inverse transport equation with error analysis}
\author{Qin Li} 
\address{Mathematics Department, University of Wisconsin-Madison, 480 Lincoln Dr., Madison, WI 53705 USA.}
\email{qinli@math.wisc.edu}
\author{Ruiwen Shu} 
\address{Mathematics Department, University of Wisconsin-Madison, 480 Lincoln Dr., Madison, WI 53705 USA.}
\email{rshu2@wisc.edu}
\author{Li Wang} 
\address{Department of Mathematics, Computational and Data-Enabled Science and Engineering Program, State University of New York at Buffalo, 244 Mathematics Building, Buffalo, NY 14206 USA.}
\email{lwang46@buffalo.edu}
\thanks{The work of R. S. is supported in part by the National Science Foundation under the grant DMS-1522184 and DMS-1107291: RNMS KI-Net. Q. L. is supported in part by a start-up fund from UW-Madison and National Science Foundation under the grant DMS-1619778 and DMS-1107291: RNMS KI-Net. The work of L.W. is supported in part by a start-up fund from SUNY Buffalo and the National Science Foundation under the grant DMS-1620135.}

\begin{document}
\maketitle

\begin{abstract}
The inverse radiative transfer problem finds broad applications in medical imaging, atmospheric science, astronomy, and many other areas. This problem intends to recover the optical properties, denoted as absorption and scattering coefficient of the media, through the source-measurement pairs. A typical computational approach is to form the inverse problem as a PDE-constraint optimization, with the minimizer being the to-be-recovered coefficients. The method is tested to be efficient in practice, but lacks analytical justification: there is no guarantee of the existence or uniqueness of the minimizer, and the error is hard to quantify. In this paper, we provide a different algorithm by levering the ideas from singular decomposition analysis. Our approach is to decompose the measurements into three components, two out of which encode the information of the two coefficients respectively. We then split the optimization problem into two subproblems and use those two components to recover the absorption and scattering coefficients {\it separately}. In this regard, we prove the well-posedness of the new optimization, and the error could be quantified with better precision. In the end, we incorporate the diffusive scaling and show that the error is harder to control in the diffusive limit.  

\end{abstract}

\section{Introduction}
Radiative transfer equation (RTE) describes the dynamics of (photon) particles in materials with various optical properties. It has been used in atmospheric science, medical imaging and many other areas as a basic model. The equation can take different forms, depending on the degrees of generality. Among them, a stationary, frequency independent form reads:
\begin{equation} \label{eqn:000}
v\cdot\nabla_xf + \sigma(x) f = \int k(x,v,v')f(x,v')\rd{v'} \,,
\end{equation}
where $f(x,v)$, defined on phase space, is the distribution of particles at location $x$ and with velocity $v$. Here $x\in\Omega\subset\RR^d$ with $d=2,3$, and $v\in V = \Sp^{d-1}$, the {\it unit} sphere in $\RR^d$. $k(x,v,v')$ is termed the scattering coefficient, representing the probability of particles moving in direction $v'$ changing to direction $v$ at location $x$. $\sigma(x)$ is the total absorption coefficient that represents certain amount of photon particles being absorbed by the material. Here we assume that $\sigma$ has no velocity dependence. The boundary is separated into an ``outgoing" and an ``incoming" part by defining: 
\begin{equation}\label{eqn:boundary}
\Gamma_\pm = \{(x,v): x\in\partial\Omega\,, \pm v\cdot n_x >0\}\,,
\end{equation}
where $n_x$ is the normal direction pointing out of $\Omega$ at $x\in\partial\Omega$. In this way, $\Gamma_-$ collects all boundary coordinates that represent particles coming into the domain where $\Gamma_+$ collects the opposite. For the wellposedness of RTE, we require inflow boundary condition, i.e. the data imposed on the ``incoming" part of $\Gamma_-$:
\begin{equation} \label{eqn:BC}
f|_{\Gamma_-} = f_-(x,v)\,, \quad (x,v) \in \Gamma_-\,.
\end{equation}

In many applications, light is sent to a bulk of material with unknown absorption and scattering properties, and light current propagating out of the material is measured. Scientists need to adjust the sources and measurement locations for recovering the material properties. This technique is used in medical imaging where near infrared light (NIR) is sent into biological tissues for tumor or bone structure~\cite{Hielscher_rheumatoid1,Hielscher_rheumatoid2}; it is also used in outer space studies: during Galileo's travel around Jupiter, pictures are taken by the near infrared mapping spectrometer (NIMS), and scientists recover components of atmosphere on each satellite by inverting RTE, through which they found that Io is covered by $SO_2$ mainly~\cite{Doute_mapping_so2}.

We study these problems from mathematical and computational side. Mathematically, we typically assume that no prior information on $\sigma$ and $k$ is known, but the entire incoming to outgoing map is given. This map is termed the {\it albedo} operator:
\begin{equation}\label{eqn:albedo}
\mathcal{A}l: \qquad f_- \mapsto f\big|_{\Gamma_+}\,.
\end{equation}
Then the goal of inverse RTE is to recover $\sigma$ and $k$ through the albedo operator. 

The wellposedness of this problem was considered in~\cite{ChoulliStefanov}, in which the authors showed 
that, given the albedo operator ~\eqref{eqn:albedo}, a full recovery of both $\sigma$ and $k$ is possible in 3D whereas in 2D, only $\sigma$ is recoverable. Some following up studies include: utilizing the Born series for the recovery~\cite{Machida_Schotland_inverse}; the illposedness of the problem if the operator's output is changed to flow current (having no velocity angle information)~\cite{Bal08,Bal09}; the passage to the illposedness in fluid regime~\cite{ChenLiWang}; and studies on various scenarios~\cite{Bal_Tamasan_sourse,BalChungSchotland}. Most of these analytical studies use the technique termed ``singular decomposition" invented in~\cite{ChoulliStefanov}. In that paper, the authors separate the data according to the singularities of different components in the measurement, each of which is in charge of recovering one property. See also a review~\cite{Bal_review}. However, despite its effectiveness in analysis, the singular decomposition idea barely sees its direct use in computation: it is unknown if the process could be repeated numerically or on real experiment, let alone the error analysis it induces in practice.

From the computational side, the topic has been extensively studied in many scenarios as well ~\cite{Arridge_couple,Arridge_piecewise_const,Ren_fPAT,Cheng_Gamba_Ren_Doping, SRH05}. Reviews could be found in Arridge~\cite{Arridge99,Arridge_Schotland09} and Ren~\cite{Ren_review}. One typical formulation is to first rewrite the equation into an optimization form and then run optimization algorithms for the recovery. More specifically, one samples $N_x$ and $N_v$ grids for $x$ and $v$ respectively, and writes the equation in the discrete form:
\begin{equation*}
(\Amat +\Bmat)\cdot\vec{f} = \vec{0}\,,\quad\text{with}\quad \vec{f}|_{\Gamma_-} = \vec{f}_-\,.
\end{equation*}
Here $\vec{f}$ is the solution sampled on all grid points:
\begin{equation} \label{fvector}
\vec{f}=[f(x_1,v_1)\,,f(x_2,v_1)\,,\cdots\,,f(x_n,v_1)\,,f(x_1,v_2)\,,\cdots\,, f(x_{N_x},v_2)\,,\cdots\,, \cdots\,,f(x_{N_x},v_{N_v})]^t\,,
\end{equation}
and $\vec{f}_-$ is $f_-$ evaluated on grid points on $\Gamma_-$. Considering the dimension of $x$ and $v$, the subscript $i$ and $j$ can be multi-indexed.
\begin{equation}\label{eqn:A_Bmat}
-\Amat = \Vmat\otimes \nabla_x+ \Imat \otimes \sigma ,\quad\Bmat = \Sigma_k \,,
\end{equation}
are discrete version of the transport and the scattering operator, where $\Vmat$ is a diagonal matrix of size $N_v \times N_v$ with diagonal elements $v_i$, and $\nabla_x$ is an $N_x \times N_x$ finite difference matrix in $x$ (depending on the scheme one uses). $\sigma$ is an $N_x \times N_x$ matrix with diagonals $\sigma(x_i)$ and $\Sigma_k$ is a block matrix with $N_x$ blocks, each block is of size $N_v \times N_v$.

Given several rounds of experiments with $\{\vec{f}^{(i)}_-\,,i = 1,\cdots N_I\}$ as the inflow on $\Gamma_-$, and the measured data $\{\phi^{(i)}\,,i=1,\cdots,N_I\}$ as the outflow on $\Gamma_+$, the typical set-up of the numerical inverse problem is to perform the following optimization problem:
\begin{equation}\label{opt-old}
\left\{\begin{array}{c}
\min_{\sigma, \Sigma_k, \{\fvec^{(i)} \} } \sum_{i} \|  \Emat_+ \fvec^{(i)}  - \vec{\phi}^{(i)}  \| + \text{regularization}
\\  {\text s.t.} \quad (\Amat + \Bmat)\cdot \vecf^{(i)}  = \vec{0}\,,\quad \Emat_-\vec{f}^{(i)} = \vec{f}_-^{(i)}\,,\quad\forall i
\end{array} \right. \,.
\end{equation}
Here the superscript $i$ denotes different experiments, and $\Emat$ is the confining operator: 
\begin{equation} \label{eqn:Emat}
\Emat_\pm\vec{f} = \vec{f}|_{\Gamma_{\pm}}\,.
\end{equation}

One advantage of this approach is that it is very straightforward, and the regularization could be adjusted to fit a priori information (for example, TV norm used on $\sigma$ for piecewise constant cases). Disadvantage is obvious as well, as mentioned in~\cite{Ren_review}: on one hand, it is unknown that whether the minimizer exists, or is unique, and the problem tends to be either overdetermined or underdetermined; on the other hand, the computational size is huge. There are $N_x$ unknowns in $\sigma$ and $N_xN_v$ unknowns in $\Sigma_k$, and in $d=3$ case, $N_x$ is about $N^3$ and $N_v$ is about $N^2$, with $N$ being the number of grid point per direction. It is extremely expensive to update even one iteration in the optimization problem. Multiple strategies are invented as modifications for better efficiency, such as utilizing the diffusion approximation~\cite{Arridge_diffusion,Arridge_diffusion_error}, linearization~\cite{Ren_review}, or using gradient instead of Jacobian for the updating~\cite{Arridge_gradient}. However, despite all the effort, it is nevertheless a pity that the wellposedness results from the analysis side is not benefitting the computation, and it is extremely hard to quantify the error of any of these methods.

In this paper, we intend to fill in the {\it gap} between analysis and computation. More specifically, we will design an algorithm that 1) is efficient, and 2) leverages as much analysis results as possible. This allows us to spell out the well-posedness and the error analysis in an exact fashion numerically. Our idea is based on the singular decomposition analysis, and we will numerically separate the three components in the measurements, using one to recover $\sigma$ and another for $k$.

More precisely, consider a concentrated incoming data $f_-$, let $\phi(x,v)$ be the solution of \eqref{eqn:000} confined on $\Gamma_+$: $\phi= f\big|_{\Gamma_+}$, then analysis in \cite{ChoulliStefanov} tells that $\phi = \phi_1 + \phi_2 + \phi_3$ with its components enjoying different singularities and thus could be separated from each other. Specifically, with $\phi_1$ separated from the rest, it could be used to recover $\sigma$:
\begin{equation*}
\mathcal{R}[\sigma] = \int_0^{\tau_-(x,v)} \sigma (x-sv) ds = \ln \left( \frac{  f_-(x-\tau_-(x,v)v,v)  }{\phi_1(x,v)}\right) \,,
\end{equation*}
where $\tau_-$ is the time needed for a non-scattering photon passing through $\Omega$ and emitting at $(x,v)\in\Gamma_+$. $\mathcal{R}$ is the X-ray transform that has been proved to be reversible. We repeat this procedure numerically. Denote $\phi_{R,1}$ the first component that gets extracted numerically from $\phi$, we show that (Theorem~\ref{thm:separation})
\begin{equation} \label{eqn:718}
\frac{|\phi_{R,1} - \phi_1|}{ |\phi_1|} = \mathcal{O}(\eps_1^{1-\delta} \eps), \quad \text{ for any}~ \delta >0\,,
\end{equation}
with $\eps_1$ and $\eps$ standing for the width of the concentrating inflow data and outflow measurement. Consequently, the numerically recovered $\sigma$ has the following the error estimate (Theorem~\ref{error-sigmaa})
\begin{equation*}
\| \sigma - \sigma^{dis} \|_2 \lesssim (\eps_1^{1-\delta}\eps + \Delta x ^2)^{1/2}\,,
\end{equation*}
where $\sigma^{dis}$ is the true absorption coefficient evaluated at grid points, and $\Delta x$ is the discretization in $x$. Here we only analyze the recovery of $\sigma$. We believe that similar analysis for the scattering coefficient $k$ can be done, but it will be much more involved and we leave it to future work. In the end of this paper, we also study the inverse RTE in diffusion regime using this approach. When the RTE can be well approximated by the diffusion equation, one gains big error in the data separation step \eqref{eqn:718}, and this error propagates in recovering $\sigma$. The whole scheme therefore breaks down. 

The rest paper is organized as follows. In the next section, we recall the singular decomposition theory used in proving the well-posedness of the inverse problem, and make an analogy in the discrete setting. In Section 3, we set up the new algorithm and provide details in the implementation. Section 4 is devoted to the error analysis, which consist of two major parts---error in the data separation and inversion. In Section 5, we introduce a diffusive scaling to the RTE and revisit the algorithms and error analysis in the presence of multiple scales.

\section{Singular decomposition}
The base of our algorithm is the singular decomposition to the measured data. In this section, we first review this technique developed in \cite{ChoulliStefanov}, and then extend it in the discrete setting that we will be working on. Let us denote 
\begin{equation}\label{eqn:notation1}
\opA = -v \cdot \nabla_x - \sigma(x), \quad \opB f = \int k(x,v,v') f(v') dv' \,,
\end{equation}
then equation \eqref{eqn:000} rewrites $(\opA + \opB)f =0$. One sees that $\opA$ consists of a free transport and damping whereas $\opB$ encodes the scattering. We let 
$$\tau_\pm (x,v) = \min \{ t \geq 0, ~x \pm tv \in \partial \Omega \}, \qquad \tau = \tau_-+ \tau_+\,,$$ 
be the time for a free transport of photon located at $x$ with velocity $v$ to travel out of $\Omega$ forward or backward. 
We also assume that 
$$\sigma(x) - \int k(x,v,v') dv' \geq \nu > 0, \quad \text{for } a.e.~ (x,v) \in \Omega \times V\,,$$
in which case the forward problem \eqref{eqn:000} is well-posed, and $(\sigma, k)$ pair is called admissible. 

\subsection{Continuous setting}
In the continuous setting, the singular decomposition is proposed in \cite{ChoulliStefanov}. The idea is that when the incoming data is concentrated around one point in $\Omega \times \Gamma_-$ space, the solution to \eqref{eqn:000} can be decomposed into three parts that enjoy different degrees of singularities, wherein the leading two singular terms can be used to recover $\sigma$ and $k$ respectively. Indeed, formally one can write the solution to 
\begin{equation} \label{eqn:001}
\left\{ \begin{array}{cc}
(\opA + \opB )f = 0 & (x,v)\in \Omega \times V
\\ f\big|_{\Gamma_-} = f_-(x,v)\, & 
\end{array} \right.
\end{equation}
as 
\begin{eqnarray}
f &=& (I + \opA ^{-1} \opB)^{-1} \opJ f_-  \nonumber
\\ &=& \opJ f_- - (\opA^{-1} \opB) \opJ f_- + (I + \opA^{-1}\opB)^{-1} (\opA^{-1}\opB)^2 \opJ f_-\, \nonumber
\\ & := &  f_1 + f_2 + f_3\,, \label{eqn:decomp1}
\end{eqnarray}
where $\opJ : \Gamma_- \rightarrow \Omega\times V$ is the solution to the pure transport and damping, i.e., 
\begin{equation}\label{eqn:define_opJ}
\opA(\opJ f_-) = 0, \quad \opJ f_- \big|_{\Gamma_-} = f_-\,.
\end{equation}
It maps the boundary condition to the entire $\Omega\times V$, with an explicit form:
\begin{equation} \label{eqn:opJ}
\opJ f_-(x,v) = e^{-\int_0^{\tau_-(x,v)} \sigma(x-sv)\rd{s}} f_-(x-\tau_-(x,v)v,v) \,.
\end{equation}
The inverse of $\opA$, denoted as $\opA^{-1}:\Omega\times V\rightarrow \Omega\times V$ has the form:
\begin{equation*}
\opA^{-1}f = -\int_0^{\tau_-(x,v)}e^{-\int_0^t \sigma(x-sv)\rd{s}} f(x-tv,v)\rd{t}\,.
\end{equation*}
It satisfies $\opA(\opA^{-1}f) = f$ with $(\opA^{-1}f)|_{\Gamma_-}=0$.

From \eqref{eqn:decomp1}, one sees that $f_1$ represents the solution with pure absorption, $f_2$ denotes solution after one scattering, and $f_3$ collects the rest.

The solutions to \eqref{eqn:001}, when confined on the boundary $\Gamma_+$, is the outgoing data that we measure. Let us denote it by $\phi$, then 
\begin{equation}\label{eqn:define_phi}
\phi = f|_{\Gamma_+} \quad\text{and}\quad \phi_i = f_i|_{\Gamma_+}\,,\quad i = 1, 2, 3\,.
\end{equation}
Immediately
\begin{equation*}
\phi = \sum_{i=1}^3\phi_i = \opA l[f_-]\,.
\end{equation*}

We now study the structure of $\opA l$. For that we use a delta function as the incoming data for $f_-$, and have the following theorem:
\begin{theorem}[\cite{ChoulliStefanov}]\label{thm:f}
Assume that $(\sigma, k)$ is admissible. Then the solution to \eqref{eqn:000} with
\begin{equation} \label{eqn:f-}
f_-(x,v) = \delta (x-x') \delta (v-v'), \quad (x',v') \in \Gamma_-
\end{equation}
has the decomposition $f(x,v;x',v') = f_1 + f_2 + f_3$, where
\begin{eqnarray}
&&f_1 = |n(x')\cdot v'| \int_0^{\tau_+(x',v')} e^{-\int_0^{{\tau_-}(x,v)} \sigma (x-sv)ds} \delta(x-x'-tv) \delta(v-v') \rd t \,; \label{eqn:f1}
\\&& f_2 =  |n(x')\cdot v'| \int_0^{{\tau_-}(x,v)} \int_0^{\tau_+(x',v')} e^{-\int_0^s \sigma(x-pv) dp} e^{-\int_0^{\tau_-(x-sv,v')}} \sigma(x-sv-pv', v')  \rd p \nonumber
\\ && \hspace{6cm} \times k(x-sv,v',v) \delta(x-x'-sv-tv') \rd t \rd s 
\\  &&(\min\{\tau, \lambda\})^{-1} |n(x') \cdot v'|^{-1} f_3 \in L^\infty (\Gamma_-; \mathcal{W})\,.
\end{eqnarray}
Here $\lambda \geq 0$ is an arbitrary constant, and 
\[
\mathcal{W} = \{f: f\in L^1(\Omega \times V),  v\cdot\nabla_x f \in L^1(\Omega \times V) \}\,.
\]
\end{theorem}

The albedo operator only takes the information on $\Gamma_+$, and we thus write the distribution kernel $\alpha_i(x,v, x',v')$ as a confinement
\[
\alpha_i(x,v;x',v') := f_i(x,v;x',v') \big|_{(x,v) \in \Gamma_+}, \quad (x',v') \in \Gamma_- \,, \qquad i = 1, 2, 3\,.
\]
It maps $(x',v')\in \Gamma_-$ to $(x,v)\in \Gamma_+$, and could be explicitly expressed:
\begin{theorem}[\cite{ChoulliStefanov}] \label{thm:alpha}
Assume that $(\sigma, k)$ is admissible. Then $\alpha(x,v;x',v') = \alpha_1 + \alpha_2 + \alpha_3$, where
\begin{eqnarray}
&& \alpha_1(x,v; x',v') =  \frac{|n(x')\cdot v'|}{n(x)\cdot v}   e^{-\int_0^{\tau_-(x,v)} \sigma(x-sv) \rd s} \delta(x-x'-\tau_+(x',v')v') \delta(v-v')\,; \label{alpha1}
\\ && \alpha_2(x,v;x',v') = \frac{|n(x')\cdot v'|}{n(x) \cdot v} \int_0^{\tau_+(x',v')} e^{-\int_0^{\tau_+(x'+tv',v)} \sigma(x-pv)\rd p} e^{-\int_0^t \sigma(x+pv')\rd p} \nonumber
\\ && \hspace{6cm} \times k(x+tv',v',v) \delta(x-x'-tv'-\tau_+(x'+tv',v)v) \rd t \,;\label{alpha2}
\\ && \min \{\tau(x',v'), \lambda   \}^{-1} |n(x')\cdot v'|^{-1}\alpha_3 \in L^\infty(\Gamma_-; L^1(\Gamma_+, \rd\xi))\,. \label{alpha3}
\end{eqnarray}
\end{theorem}

We conclude that, for any incoming data $f_-(x,v)$:
\begin{equation}\label{eqn:phi-alpha}
\phi(x,v) = \opA l[f_-](x,v) = \sum_{i=1}^3\phi_i = \int_{\Gamma_-}\sum_{i=1}^3\alpha_i(x,v;x',v')f_-(x',v')\rd\xi(x',v')\,,
\end{equation}
where $\rd\xi(x',v') = |n(x') \cdot v| \rd\mu(x')\rd v $ is the measure on $\Gamma_-$, and the three components have very different singularities:
\begin{itemize}
\item[$\alpha_1$] is a delta function in both $x$ and $v$. It belongs to $L^1(\Omega\times V)$, and contains information only from $\sigma$ but not $k$.
\item[$\alpha_2$] is an integration of a delta function over a one dimensional manifold. The two exponentials reflect the particle traveling from $x'$ to $x'+tv'$ and from $x'+tv'$ to $x$ respectively. The particle changes its velocity from $v'$ to $v$ at $x'+tv'$ with the probability $k(x+tv',v',v)$. This term encodes the information of particles who travel and change directions once.
\item[$\alpha_3$] collects of all the rest information.
\end{itemize}

\subsection{Discrete setting}
The same formulation can be written down on the discrete level. Using the notation from the introduction, we write the equation, incorporating the boundary conditions:
\begin{equation} \label{eqn:100}
(\Imat +\Amat ^{-1}\Bmat) \vecf = \Amat^{-1}\Jmat_0 \vecf_-\,,
\end{equation}
where $\vecf$ is defined in \eqref{fvector}. Here $\Jmat_0$ numerically resembles $\opJ$. Denote $N_{b\pm}$ the number of grid points $(x,v)$ on $\Gamma_\pm$, then $\Jmat_0$ is a matrix of size $N_xN_v\times N_{b_-}$. 

Using the Neumann expansion:
\begin{align*}
\left(\Imat+\Xmat\right)^{-1} &= \Imat - \left(\Imat+\Xmat\right)^{-1}\Xmat\,,\quad \nonumber\\
&= \Imat - \left(\Imat - \left(\Imat+\Xmat\right)^{-1}\Xmat\right)\Xmat\, \nonumber \\
&= \Imat - \Xmat + \left(\Imat+\Xmat\right)^{-1}\Xmat^2\,, \label{eqn:101}
\end{align*}
we let $\Xmat = \Amat^{-1}\Bmat$ and separate the solution of \eqref{eqn:100} into:
\begin{eqnarray}
\vecf &=& \Amat^{-1} \Jmat_0 \fvec_- - \Amat ^{-1}\Bmat \Amat^{-1} \Jmat_0 \vecf_-  + (\Imat + \Amat^{-1}\Bmat)^{-1} (\Amat^{-1} \Bmat )^2  \Jmat_0\fvec_- \nonumber
\\ &:=& \vecf_1 + \vecf_2 + \vecf_3\,.  \label{eqn:1011}
\end{eqnarray}

Comparing it with~\eqref{eqn:decomp1}, we see that the three vectors are simply counterparts of $f_i$. As suggested by Theorem~\ref{thm:f}, these three vectors should have different sparsities. Similar to the discussion for the continuous setting, here we see that $\vec{f}_1$ includes information on $\Amat$ only, which could be used to recover $\sigma$, while $\vec{f}_2$ takes up information from $\Bmat$ that is equivalent to $\Sigma_k$.

\section{Numerical algorithm}
As mentioned in the introduction, most of the currently available algorithms are based on optimization, and they typically write as
\begin{equation}\label{opt-old}
\left\{\begin{array}{c}
\min_{\sigma, \Sigma_k, \{\fvec^{(i)}\}} \sum_{i} \|  \Emat_+ \fvec^{(i)}  - \vec{\phi}^{(i)}  \| + \text{regularization}
\\  {\text s.t.} \quad (\Amat + \Bmat) \vecf^{(i)}  = \mathbf{0}\,, \quad \vecf^{(i)}\big|_{\Gamma_-} = \vecf_-^{(i)}, \quad i = 1,2 \cdots, N_I\,,
\end{array} \right.
\end{equation}
where the superscript $i$ denotes different rounds of experiments, $N_I$ is the total number of experiments conducted, and $\Emat_+$ is defined in \eqref{eqn:Emat}. The approach is straightforward, but it is lack of analytical justification: there is no guarantee that the minimizer exists and will be unique, nor does it tell how to choose the correct regularization, and what will the error be. What is more, $\sigma$ and $\Sigma_k$ are recovered simultaneously which requires a lot of computation in each optimization iteration step.

In this section, we set up a new optimization framework in recovering $\sigma(x)$ and $k(x,v,v')$ {\it separately}. As indicated by the singular decomposition method from~\cite{ChoulliStefanov}, the measurement could be separated into three parts based on the different regularities they enjoy, and the first two terms encode information for $\sigma$ and $k$ respectively. Based on this, we propose a new way of the recovery, and this new approach comes with more rigorous error quantification.
 
\subsection{Algorithm set up}\label{sec:setup}
We first write down the algorithm in the continuous sense, following the ideas in~\cite{ChoulliStefanov}. From here on, we will assume that the experiments are well set in the sense that the measurement is placed at the boundary where free transport photons emit, corresponding to the input stimulus. Then the algorithm reads as follows.
\begin{itemize}
\item[] \underline{Algorithm (continuous)}
\item[] Input: concentrated source $f_-(x',v')$, $\forall ~(x',v') \in \Gamma_-$; measurement $\phi(x,v)$, $\forall (x,v) \in \Gamma_+$
\item[] Output: $\sigma(x)$, $k(x,v,v')$
\item[] Step 0) Decompose data $\phi = \phi_1+ \phi_2 + \phi_3$;
\item[] Step 1) Recover $\sigma(x)$ by solving the following problem
\begin{equation}\label{eqn:opt-cont1}
\min_{\sigma} \mathcal{F}\left(\mathcal{R}[\sigma](x,v) - a(x,v)\right)\,;
\end{equation}
\item[] Step 2) Recover $k(x,v,v')$ by solving 
\begin{equation} \label{eqn:opt-cont2}
\left\{ \begin{array}{c}
\min_{k, f}  \mathcal{F} \left( f\big|_{\Gamma_+} - \phi_1 -\phi_2  \right)
\\ {\text{s.t.}}~  (\opA(\sigma) + \opB) f = 0, ~f\big|_{\Gamma_-} = f_-
\end{array} \right.\,.
\end{equation}
\end{itemize}

In the problem we formulated, $\mathcal{F}$ is a {\it nonnegative convex} fit-to-data function. $\mathcal{R}$ is the X-ray transform:
\begin{equation*}
\mathcal{R}[\sigma](x,v) = \int_0^{\tau_-(x,v)} \sigma(x-sv)\rd{s}\,,\quad\forall (x,v)\in\Gamma_+\,,
\end{equation*}
and $a$ is calculated from the data:
\begin{equation}\label{eqn:def_cont_a}
a(x,v) = \ln \left(  \frac{f_-(x-\tau_-(x,v)v,v)}{\phi_1(x,v)}   \right)\,.
\end{equation}

To justify the validity of this algorithm, we note that:
\begin{itemize}
\item[Step 0] can be done due to the different singularities of $\phi_i$ according to Theorem~\ref{thm:alpha}, once the incoming data is made concentrated.
\item[Step 1] is written in an optimization form for later convenience, but in fact, the minimum could be achieved and is zero. Indeed, according to the definition in~\eqref{eqn:define_phi} and~\eqref{eqn:opJ}, one immediately sees that for all $(x,v)\in\Gamma_+$, given $x-x'  \parallelsum v$ and $v=v'$:
\begin{equation}\label{eqn:x_ray}
\mathcal{R}[\sigma](x,v) = \int_{0}^{\tau_-(x,v)}\sigma(x-sv)\rd{s} = \ln \left[ f_-(x',v')/ \phi_1(x,v)\right]\,,
\end{equation}
which is the same as $a(x,v)$ in \eqref{eqn:def_cont_a}.
\item[Step 2] is also written in the optimization form. From the theory in \cite{ChoulliStefanov}, a unique recovery of $k$ is available once $\sigma_a$ is obtained from the first step. Therefore, the minimum could be achieved and is zero.  
\end{itemize}

\begin{remark}
Another straightforward solver is to replace the optimization problem in Step 1 by:
\begin{equation} \label{eqn:opt-cont1-replace}
\left\{ \begin{array}{c}
\min_{\sigma, f} \mathcal{F} \left( f\big|_{\Gamma_+}  - \phi_1  \right)
\\ {\text{s.t.}}~  \opA f = 0, ~f\big|_{\Gamma_-} = f_-
\end{array} \right.\,;
\end{equation}
However, as we can see $\sigma$ here is involved in a nonlinear way, making the optimization problem harder to analyze.
\end{remark}
The same procedure could be taken in the discrete setting for numerical simulation.

\vspace{0.4cm}
\begin{itemize}
\item[] \underline{Algorithm (discrete)}
\item[] Input: concentrated source $\vecf_-^{(i)}$, concentrating around $(x^{(i)},v^{(i)}) \in \Gamma_-$; measurement $\phi^{(i)}$. $i=1,\cdots N_I$.
\item[] Output: $\Sigma$, $\Sigma_k$
\item[] Step 0) Decompose data $\phi^{(i)} = \phi_{R,1}^{(i)} + \phi_{R,2}^{(i)}+ \phi_{R,3}^{(i)}$ for all $i$\,;
\item[] Step 1) Recover $\sigma(x)$ by solving the following problem
\begin{equation}\label{eqn:opt-dis1}
\min_{\Sigma} \mathcal{F}\left(\Rmat\cdot\Sigma - \vec{a}\right)\,;
\end{equation}
\item[] Step 2) Recover $\Sigma_k$ by solving 
\begin{equation}\label{eqn:opt-dis2}
\left\{ \begin{array}{c}
\min_{\Sigma_k}  \mathcal{F} \left( \left\{  \Emat_+ \fvec ^{(i)} - \phi_1^{(i)} - \phi_2^{(i)} \right\}_{i=1,2,\cdots, N_I} \right)
\\ {\text{s.t.}}~  (\Amat(\Sigma) + \Bmat) \vecf^{(i)} =0, \quad \vecf^{(i)} \big|\Gamma_- =  \vecf_-^{(i)}, \quad i = 1, 2, \cdots, N_I
\end{array} \right.\,.
\end{equation}
\end{itemize}
Here the subindex $R$ in Step 0 indicates the numerical recovery. $\mathcal{F}$ is the discrete version of the fit-to-data function, and $\Rmat$ is the numerical integration of X-ray transform in~\eqref{eqn:x_ray}, with each of its row representing one experiment. Vector $\vec{a}$ consists of data collected at specific grid point:
\begin{equation}\label{eqn:def_dis_a}
\vec{a}_j = \ln \left( \frac{\vec{f}^{(j)}_-(x^{(j)}, v^{(j)})}{\phi^{(j)}_{R,1}(x_*^{(j)}, v^{(j)}_*)}   \right)  \,,
\end{equation}
where $(x^{(j)}, v^{(j)}) \in \Gamma_-$, and its counterpart denoted as $(x^{(j)}_*, v^{(j)}_*)$ takes the form
\begin{equation} \label{xv_star}
x_*^{(j)} = x^{(j)} + \tau_+(x^{(i)}, v^{(i)}) v^{(i)}, \quad v_*^{(j)} = v^{(j)} \,.
\end{equation}
As written, these steps are pure resemblance of the algorithm in the continuous setting, and each step requires a specially designed implementation, to ensure the wellposedness, and controllable error. We discuss the implementation in the following subsection, and the error analysis is left to Section~\ref{sec:error}.

\begin{remark}
An immediate advantage of our new formulation \eqref{eqn:opt-dis1} over the conventional one \eqref{opt-old} is the size reduction: instead of looking for $\Sigma$ and $\Sigma_k$ simultaneously, which is a problem of size $(N_xN_v+ N_x)^2$, we find $\Sigma$ first and then find $\Sigma_k$, and the former one is of a much reduced size. 
\end{remark}

\subsection{Implementation}
In this section, we will make clear how each of those steps in the discrete algorithms can be performed. 
\subsubsection{Decomposition}
Given a concentrated incoming data $f^{(i)}_-$ on $\Gamma_-$, one could collect the outgoing data $\phi^{(i)}$ on $\Gamma_+$, which analytically can be separated into three parts $\phi^{(i)}_{1/2/3}$. Numerically, however, it is not possible to conduct the separation exactly. Instead, we obtain the recovered data, denoted as $\phi^{(i)}_{R,1/2/3}$. Therefore, we need to find a way to define $\phi^{(i)}_{R,j}$ that is simple to obtain and close to $\phi^{(i)}_j$ enough with a small error. To this end, let us first assume that $f^{(i)}_-$ is concentrated around $(x^{(i)},v^{(i)})$, with the width smaller than $\Delta x$ and $\Delta v$. Therefore, $\vec{f}^{(i)}_-$ has only one nonzero value located at $(x^{(i)},v^{(i)})$ grid. Then we simply set
\begin{eqnarray}
&& \phi_{R,1}^{(i)}(x,v) = \begin{cases} \phi(x,v)\,,\quad & x = x^{(i)}_*, ~ v = v^{(i)}_* \\ 0\,,\quad&\text{elsewhere}\end{cases}\,,  \qquad (x^{(i)}, v^{(i)})\in \Gamma_-, \quad (x^{(i)}_*, v^{(i)}_*) \in \Gamma_+\,; \label{eqn:phi_1_recover} \\
&&   \phi_{R,2}^{(i)}(x,v) = \begin{cases} \phi(x,v)\,,\quad &\exists (s,t)~\text{s.t.}~ x-sv =x^{(i)} + tv^{(i)},  \quad v \neq v^{(i)}\\ 0\,,\quad &\text{elsewhere}\end{cases}\, ;\nonumber\\
&& \phi_{R,3}^{(i)} = \phi - \phi_{R,1} -\phi_{R,2}\,.\nonumber
\end{eqnarray}

\subsubsection{Recovering $\Sigma$}
To recover $\sigma$ from $\phi_1$, one just need to conduct an inverse X-ray transform as displayed in~\eqref{eqn:x_ray}. Since the X-ray transform has explicit inversion formula (will be detailed below), and it is in an integral form, one way in the discrete setting is to use quadrature rules to approximate inversion formula. This requires evaluating the integrand on the grids and performing the summation. However, the process is well-known to be numerically very unstable~\cite{Bronnikov_numerics_CT,Natterer_CT_book}. To overcome this difficulty, many strategies have been invented, including the algebraic reconstruction technique, direct algebraic methods, among many others~\cite{Natterer_CT_book}. Earlier in this century, more attention has been placed on using the optimization framework instead of a direct inversion and adopting the Tikhonov regularization to overcome the large conditioning. This is the approach that we will be taking.

Specifically, in Step 1, we modify the optimization with a regularizer:
\begin{equation}\label{eqn:opt-dis1-regularizing}
\min_{\sigma}\|\Rmat\cdot\sigma - \vec{a}\|_X + \lambda\|\sigma\|_Y\,,
\end{equation}
where the first term represents the mismatch and the second term is the regularizer ensuring the error in the measurement stay controlled. Both terms are convex, and the existence and the uniqueness of the minimizer is obvious. In the next section, we will analyze the error brought by the introduction of the regularizer. 

We remark here that a more straightforward form in recovering $\sigma_a$ in our problem could be 
 \begin{equation*}
\left\{ \begin{array}{c}
\min_{\sigma} \sum_{i=1}^{N_I}\|\Emat_+ \fvec ^{(i)}- \phi_1^{(i)}\|_X  +\lambda\|\sigma\|_Y
\\ {\text{s.t.}}~  \Amat \vecf^{(i)} = 0, \quad \vecf^{(i)}\big|_{\Gamma_-} = \vecf_-^{(i)}\,, \quad i = 1, 2, \cdots, N_I\,,
\end{array} \right.
\end{equation*}
which can be considered as a numerically implementation of~\eqref{eqn:opt-cont1-replace}. However, as claimed before, here $\sigma$ is involved in the problem in a nonlinear fashion, and it is not clear why the minimizer exists, or is unique, and the error would be hard to quantify.

To end this section, we include the inversion formula for the X-ray transform for completeness. In 2D ($x,v \in \mathbb{R}^2$), the X-ray transform is equivalent to Radon transform, which admits a unique inversion formula \cite{BK08}:
\begin{equation}\label{recon1}
\sigma(x) = \frac{1}{2\pi^2}\int _0^\pi \mathcal{R}[\sigma(\cdot, \theta)\ast h](x_1\cos\theta + x_2\sin\theta) \rd\theta\,.
\end{equation}
Here $h$ is the inverse Fourier transform of $|k|$, and $\mathcal{R}[\sigma]$ is defined in \eqref{eqn:x_ray}.
For dimension higher than two, the X-ray transform is different from Radon transform, and one needs to first translate a series of X-ray projection into a Radon projection and then perform the inverse Radon transform~\cite{katsevich_analysis_CT,Katsevich_improved_backprojection,Katsevich_inverse_xray}. Specifically, for helices trajectory of sources, denote 
\[
\mathcal{D} (y,v) = \int_0^{\tau_-(y,v)} \hspace{-0.3cm}\sigma(y-sv)\rd s\,,
\]
then for properly chosen vector $e_\nu(p,x)$ and weight $\mu_\nu$, the reconstruction formula is 
\begin{equation} \label{recon2}
\sigma(x) = -\frac{1}{2\pi} \int_{I_{BP}(x)} \frac{I(p,x)}{|x-y(p)|}\rd p\,,
\end{equation}
where
\begin{equation*}
I(p,x)=\sum_{\nu}^{N_e} \mu_\nu \int_{-\pi}^{\pi} \mathcal{D}'(y(p), \cos\gamma b + sin\gamma e_\mu) \frac{1}{\sin \gamma} \rd \gamma\,,
\end{equation*}
and the derivative of $\mathcal{D}$ is with respect to the first variable. $I_{BP}(x)$ is the back-projection interval \cite{BK08}. Here in either cases, we see that analytically a unique reconstruction of $\sigma(x)$ is available. 

\section{Error analysis}\label{sec:error}
This section is devoted to analyzing the reconstruction error $\|\Sigma - \sigma^{dis}_a\|_2$, where $\Sigma$ is obtained from solving \eqref{eqn:opt-dis1-regularizing} with $\phi_{R,1}$ given in \eqref{eqn:phi_1_recover}. $\sigma^{dis}_a$ is the true media sampled on the grid points, with the superscript ``dis" indicating that it is the discrete version. The analysis below is confined in 3D.  

In the recovery for $\Sigma$, two steps are taken: the separation of data and the minimization for inverse X-ray transform. We cumulatively analyze them:
\begin{itemize}
\item[1)] data separation: to extract $\phi_1^{(i)}$ from the measurement $\phi^{(i)}$, some assumptions have been made, and we need to study $\| \phi_{R,1}^{(i)} - \phi_1^{(i)}\|$, the distance between the recovery \eqref{eqn:phi_1_recover} and the true data;
\item[2)] determine $\Sigma$ from \eqref{eqn:opt-dis1-regularizing} using discrete reconstruction formula. The regularization has been added to control the error from 1) but it inevitably introduces the regularizing error.
\end{itemize}
We examine each error closely in the following two subsections.

\subsection{Study of $\phi_{R,1} - \phi_1$}
According to~\eqref{eqn:phi_1_recover}, incoming data is placed at $(x_0,v_0)$ and $\phi_{R,1}$ is defined to be zero except for a particular point---$(x_0 + \tau_+(x_0,v_0)v_0, v_0)$---the counterpart of $(x_0, v_0)$ on $\Gamma_+$, and at this point, $\phi_{R,1}$ simply takes the value of $\phi$, with the intuition that both $\phi_2$ and $\phi_3$ have very limited contribution at this particular point. In this section we quantify the error produced by ignoring $\phi_{2/3}$'s contribution.

More precisely, assume the incoming source $f_-(x',v')=\psi \left(\frac{|x'-x_0|}{\eps} \right)\psi \left(\frac{|v'-v_0|}{\eps} \right)$ to concentrate at $(x_0,v_0)\in\Gamma_-$, and the measurement is taken in the neighborhood of its counterpart coordinate $(x_{0*}, v_{0*}):=((x_0+\tau_+(x_0,v_0)v_0,v_0)\in\Gamma_+$, i.e., 
\[
E_i =  \int_{\Gamma_+}\phi_i(x,v)\psi \left(\frac{|x-(x_0+\tau_+(x_0,v_0)v_0)|}{\eps_1} \right)\psi \left(\frac{|v-v_0|}{\eps_1} \right)\rd{\xi(x,v)}\,.
\]
Here $\eps$, and $\eps_1$ denote the concentration of the source and measurement respectively. $\psi$ is a smooth positive function supported on $[-1,1]$ with $\psi=1$ on $[-1/2,1/2]$. Then using \eqref{eqn:phi-alpha}, $E_i$ writes
\begin{equation}\label{eqn:E_i}
\begin{split}
E_i= & \int_{\Gamma_-}\int_{\Gamma_+}\alpha_i(x,v; x',v')\psi \left(\frac{|x'-x_0|}{\eps} \right)
\psi \left(\frac{|v'-v_0|}{\eps} \right) \\
& \hspace{2cm} \psi \left(\frac{|x-(x_0+\tau_+(x_0,v_0)v_0)|}{\eps_1} \right)\psi \left(\frac{|v-v_0|}{\eps_1} \right)\rd{\xi(x,v)}\rd{\xi(x',v')} \,.
\end{split}
\end{equation}
We will show that $E_1$ is much larger than $E_{2,3}$ for small $\eps$ and $\eps_1$, which implies that, at this particular point, the error $\phi_{R,1} - \phi_1 = \phi-\phi_1=\phi_2+\phi_3$ is small. In particular, we have: 
\begin{theorem}\label{thm:separation}
Consider the incoming data given by $f_-(x',v') = \psi \left(\frac{|x'-x_0|}{\eps}\right)\psi \left(\frac{|v'-v_0|}{\eps} \right)$, and $E_i$ defined in~\eqref{eqn:E_i}. Assume there exists positive constants $C_1$ such that
		\begin{equation*}
		\sigma(x) \le C_1,\quad k(x,v,v') \le C_1,\quad \forall x,v,v' \,,
		\end{equation*}
	 and $\tau_+$ is Lipschitz continuous near $(x_0,v_0)$. 
		Then there exists constants $c$, $C$, and $C_\delta$ such that
		\begin{equation*}
		E_1 \ge c \eps_1^4 ,\quad E_2 \le C\eps^4 \eps_1^2 ,\quad E_3 \le C_\delta\eps_1^{2-\delta} \eps^4
		\end{equation*}
		for any $\delta>0$.
Consequently, we have
\begin{equation}\label{eqn:E123}
\frac{E_2+E_3}{E_1} =\mathcal{O}(\eps_1^{-2-\delta} \eps^4)\, \quad \text{for any } ~ \delta >0\,,
\end{equation}
and thus the relative error is:
\begin{equation} \label{eqn:phi123}
\frac{|\phi_{R,1}-\phi_1|}{|\phi_1|}=O(\eps_1^{-2-\delta} \eps^4).
\end{equation}
\end{theorem}

To get the relationship among $E_i$s, we need to estimate their magnitudes individually. From the relation \eqref{eqn:E_i} and the expression of $\alpha_1$ and $\alpha_2$ in \eqref{alpha1} \eqref{alpha2}, $E_1$ and $E_2$ can be evaluated straightforwardly. On the contrary, $E_3$ needs more sophisticated analysis and as such, we first bound $\alpha_3$, the kernel of the third part of the albedo operator, in the following theorem.

\begin{proposition}\label{thm1}
$\alpha_3(x,v;x',v')\in L^\infty(\Gamma_-,L^p(\Gamma_+,\rd{\xi}))$ if $p<2$.
\end{proposition}
To prove this theorem, notice that $\alpha_3 = f_3 \big|_{(x,v)\in \Gamma_+}$, $f_3=(\opA^{-1} \opB)^2f$, and by Proposition 2.3 of \cite{ChoulliStefanov}, $\|f\|_{L^1(\Omega\times V)} \le C\|f_-\|_{L^1(\rd{\xi})}$, therefore we basically need the boundedness of $\opA^{-1}$ and $\opB\opA^{-1}\opB$ (Lemma~\ref{lem1.1} and \ref{lem1.3}). We will also show that the operator $\opB\opA^{-1}\opB$ could send $L^1$ data to $L^p$ (Lemma~\ref{lem2.2}). The results are summarized in the following few lemmas. 

	\begin{lemma}\label{lem1.1}
Let $g$ be a function defined on $\Omega\times V$. $1\le p < \infty$. Then $\exists C$ such that:
		\begin{equation*}
		\|\opA^{-1}g|_{\Gamma_+}\|_{ L^p(\rd{\xi})} \le C\|g\|_{L^p(\Omega\times V)}.
		\end{equation*}
	\end{lemma}
	\begin{proof}
		\begin{equation*}
		\begin{split}
		\|\opA^{-1}g|_{\Gamma_+}\|_{L^p(\rd{\xi})}^p & = \int_{\Gamma_+} \left|\int_0^{\tau_-(x,v)}e^{-\int_0^t\sigma(x-sv,v)\rd{s}}g(x-tv,v)\rd{t}\right|^p\rd{\xi(x,v)} \\
		& \le \int_{\Gamma_+} \left[\int_0^{\tau_-(x,v)}|g(x-tv,v)|\rd{t}\right]^p\rd{\xi(x,v)} \\		& \le C\int_{\Gamma_+} \int_0^{\tau_-(x,v)}|g(x-tv,v)|^p\rd{t}\rd{\xi(x,v)} \\
		& = C\|g\|_{L^p(X\times V)}^p
		\end{split}
		\end{equation*}
		where the second inequality uses H\"older inequality, and $C=\sup_{x,v}\tau_-(x,v)^{p/p'}$, $\frac{1}{p}+\frac{1}{p'}=1$.
	\end{proof}
	\begin{lemma}\label{lem1.3}
		Let $g$ be a function defined on $\Omega\times V$. Assume that $k(x,v,v')\le C_1$. If $p\ge1$ and $q<\frac{3p}{3-p}<\infty$, then
		\begin{equation*}
		\|\opB \opA^{-1} \opB g\|_{L^q(\Omega\times V)} \le C(p,q) \|g\|_{L^p(\Omega\times V)}\,.
		\end{equation*}
	\end{lemma}
	\begin{proof}
	\begin{equation} \label{eqn:716}
	\begin{split}
	(\opB \opA^{-1} \opB g)(x,v) 
	= & -\int_V k(x,v',v)\int_0^{\tau_-(x,v')}e^{-\int_0^t\sigma(x-sv')\rd{s}}(\opB g)(x-tv',v')\rd{t}\rd{v'} \\
	= & - \int_{\Omega_x}k(x,v',v)e^{-\int_0^t\sigma(x-sv')\rd{s}}(\opB g)(y,v')t^{-2}\rd{y} \\
	= & - \int_{\Omega_x}k(x,v',v)e^{-t\int_0^1\sigma((1-s')x+s'y)\rd{s'}}\int_V k(y,w,v')g(y,w)\rd{w} t^{-2}\rd{y} \\
	= & \int_\Omega\int_V K_1(x,v,y,w)g(y,w) \rd{w}\rd{y}
	\end{split}
	\end{equation}
	with the change of variable
	\begin{equation*}
	y = x-tv', \quad  \rd{y} = t^2\rd{t}\rd{v'},\quad t=|x-y|,\quad v' = \frac{x-y}{|x-y|}
	\end{equation*}
and $\Omega_x$, the integration domain of $y$, is the set of $y\in\Omega$ such that the segment from $x$ to $y$ is contained in $\Omega$.

	The integral kernel $K_1$ is given by
	\begin{equation*}
	K_1(x,v,y,w) = {\bf 1}_{\Omega_x}(y)k(x,v',v)e^{-t\int_0^1\sigma((1-s')x+s'y)\rd{s'}}k(y,w,v') |x-y|^{-2}\,.
	\end{equation*}
	Thus, by the assumption that $k\in L^\infty$, one has
		\begin{equation} \label{bound_K}
		|K_1(x,v,y,w)|\le C|x-y|^{-2}\,.
		\end{equation}
	Using this estimate, we can finish the proof by the Hardy-Littlewood-Sobolev inequality:
		\begin{equation*}
		\begin{split}
		\|\opB \opA^{-1} \opB g\|_{L^q(X\times V)}^q & \le C\|\,|x|^{-2}*_{x,v}g\|_{L^q(X\times V)}^q \\
		& = C\int_{\Omega\times V}\left|\int_{\Omega\times V}|x-y|^{-2}g(y,w)\rd{y}\rd{w}\right|^q\rd{x}\rd{v} \\
		& \le C\int_\Omega\left|\int_\Omega|x-y|^{-2}\tilde{g}(y)\rd{y}\right|^q\rd{x} \\
		& = C\|\,|x|^{-2}*_x\tilde{g}\|_{L^q(\Omega)}^q \\
		& \le C\|\tilde{g}\|_{L^p(\Omega)}^q \,,
		\end{split}
		\end{equation*}
		where $\tilde{g}(x)=\int_V g(x,v)\rd{v}$, and the last inequality uses the HLS inequality in $\Omega$, hereby imposing the restrictions on $p$ and $q$. Then notice from the H\"older inequality
		\begin{equation*}
		\|\tilde{g}\|_{L^p(\Omega)}^p = \int_\Omega \left|\int_V g(x,v)\rd{v} \right|^p\rd{x} \le C\int_\Omega\int_V|g|^p\rd{v}\rd{x} = C\|g\|_{L^p(\Omega\times V)}^p \,,
		\end{equation*}
		the result directly follows. 
	\end{proof}
		
	\begin{lemma}\label{lem2.1}
		Let $f_1$ be defined in \eqref{eqn:f1}, then for $p<2$
		\begin{equation*}
		\|\opB \opA^{-1} \opB f_1\|_{L^p(\Omega\times V)}\le C_p\,.
		\end{equation*}
	\end{lemma}
	\begin{proof}
		Recall $f_1$ 
		\begin{equation*}
		f_1(x,v) = |n(x')\cdot v'|\int_0^{\tau_+(x',v')}e^{-\int_0^{\tau_-(x,v)}\sigma(x-sv)\rd{s}}\delta(x-x'-tv)\delta(v-v')\rd{t}\,,
		\end{equation*}
		Then
		\begin{equation*}
		\begin{split}
		|\opB \opA^{-1} \opB f_1(x,v) | & \le C\int_\Omega\int_V\int_0^{\tau_+(x',v')}|x-y|^{-2}\delta(y-x'-tw)\delta(w-v')\rd{t}\rd{w}{\rd{y}} \\
		& = C\int_0^{\tau_+(x',v')} |x-(x'+tv')|^{-2}\rd{t}
		\end{split}
		\end{equation*}
	thanks to \eqref{eqn:716} and \eqref{bound_K}.

		For any $x\in X$, write the parallel and perpendicular component of $x$ with respect to $v'$ as
		\begin{equation*}
		x = x' + x_\parallel + x_\perp,\qquad \text{with} \quad x_\parallel = ((x-x')\cdot v')v'\,,
		\end{equation*}
		then one sees that 
		\begin{equation*}
		\begin{split}
		\int_{-\infty}^\infty |x-(x'+tv')|^{-2}\rd{t} & = \int_{-\infty}^\infty (|x_\parallel-tv'|^2 + |x_\perp|^2)^{-1}\rd{t} = \int_{-\infty}^\infty(||x_\parallel|-t|^2 + |x_\perp|^2)^{-1}\rd{t} \\
		& = \int_{-\infty}^\infty(|t|^2 + |x_\perp|^2)^{-1}\rd{t} = \pi |x_\perp|^{-1}\,.
		\end{split}
		\end{equation*}
		Therefore, for any $q<2$,
		\begin{equation*}
		\|\opB \opA^{-1} \opB f_1\|_{L^q}^q \le C\int_\Omega\int_V |x_\perp|^{-q}\rd{v}\rd{x} = C\int_{-R}^R\int_{|x_\perp|\le R}|x_\perp|^{-q}\rd{x_\perp}\rd{|x_\parallel|} \le C \,,
		\end{equation*}
		where $R$ is the diameter of $\Omega$, since $x_\perp$ lives in a 2d space.
	\end{proof}
	\begin{lemma}\label{lem2.2}
		Let $f$ be the solution to \eqref{eqn:000} with incoming data \eqref{eqn:f-}, then for $p< 2$ there is $C$ such that
		\begin{equation*}
		\|\opB \opA^{-1} \opB f\|_{L^p(\Omega\times V)} \le C \,.
		\end{equation*}
	\end{lemma}
	\begin{proof}
		From the previous lemma, we know that $\|\opB \opA^{-1} \opB f_1\|_{L^p(\Omega\times V)}\le C$ for any $p<2$. Then notice that $\opA^{-1}$ and $\opB$ are bounded operators on $L^p$ (which is obvious from their explicit expressions). Then, since $\opB \opA^{-1} \opB f_2 = -\opB \opA^{-1} \opB(\opA^{-1} \opB f_1)=-\opB \opA^{-1} (\opB\opA^{-1} \opB f_1)$, one gets $\|\opB\opA^{-1} \opB f_2\|_{L^p(\Omega\times V)}\le C$. Finally, from the fact that $\|f\|_{L^1(\Omega\times V)}\le C$, one gets $\|f_3\|_{L^q(\Omega\times V)} = \|\opA^{-1}\opB\opA^{-1} \opB f\|_{L^q(\Omega\times V)} \le C$ by Lemma \ref{lem1.3}, if $q<3/2$. Then using Lemma \ref{lem1.3} again give $\|\opB\opA^{-1} \opB f_3\|_{L^p(\Omega\times V)}\le C$ for any $p<3$. 
	\end{proof}
	Finally, given the fact that $f_3 = (\opA^{-1} \opB)^{2} f$, Lemma \ref{lem1.1} and Lemma \ref{lem2.2} imply Proposition \ref{thm1}. 
The proof of Theorem~\ref{thm:separation} is now in order. 

	\begin{proof}[Proof of Theorem~\ref{thm:separation}]
	Using (\ref{alpha1}), one can see that
	\begin{equation} \label{eqn:error_1}
	\begin{split}
	E_1 = & \int_{\Gamma_-}\int_{\Gamma_+}  \frac{|n(x') \cdot v'|}{n(x)\cdot v}    e^{-\int_0^{\tau_-(x,v)} \sigma(x-sv) ds} \delta(x-x'-\tau_+(x',v')v') \delta(v-v')\\
	&\qquad  \psi \left(\frac{|x'-x_0|}{\eps} \right)\psi \left(\frac{|v'-v_0|}{\eps} \right)\psi \left(\frac{|x-(x_0+\tau_+(x_0,v_0)v_0) |)}{\eps_1} \right)\psi \left(\frac{|v-v_0|}{\eps_1}\right)\rd{\xi(x,v)}\rd{\xi(x',v')}\\
	=&  \int_{\Gamma_-}  |n(x')\cdot v'|e^{-\int_0^{\tau_-((x'+\tau_+(x',v')v'),v')}\sigma((x'+\tau_+(x',v')v')-pv')\rd{p}}  \psi \left(\frac{|x'-x_0|}{\eps} \right)\psi \left(\frac{|v'-v_0|}{\eps} \right)\\
	& \qquad \psi \left(\frac{|(x'+\tau_+(x',v')v')-(x_0+\tau_+(x_0,v_0)v_0)|}{\eps_1} \right) \psi\left(\frac{|v'-v_0|}{\eps_1} \right) \rd{\xi(x',v')}  \,.
	\end{split}
	\end{equation}
	Due to the Lipschitz continuity of $\tau_+$, there exists a small constant $c<\frac{1}{2}$ such that $|x'-x_0|<c\eps_1,|v'-v_0|<c\eps_1$ implies $|(x'+\tau_+(x',v')v')-(x_0+\tau_+(x_0,v_0)v_0|<\frac{1}{2}\eps_1$. Also, since $\tau_-$ and $\sigma$ have upper bounds, the exponential term has a lower bound. Thus 
	\begin{equation*}
	|E_1| \ge \int_{|x'-x_0|<c\eps_1,|v'-v_0|<c\eps_1}\rd{\xi(x',v')} \ge c\eps_1^4\,.
	\end{equation*}
	
	To estimate $E_2$, one uses (\ref{alpha2}) to get
	\begin{equation*}
	\begin{split}
	|E_2| = & \int_{\Gamma_-}\int_{\Gamma_+}   \int_0^{\tau_+(x',v')} \frac{|n(x') \cdot v'|}{n(x)\cdot v}    e^{-\int_0^{\tau_+(x'+tv',v)} \sigma(x-pv)dp} 
	e^{-\int_0^t \sigma(x+pv')dp}k(x+tv',v',v) \\ 
	&  \qquad \delta(x-x'-tv'-\tau_+(x'+tv',v)) \rd t  \psi \left(\frac{|x'-x_0|}{\eps} \right)\psi \left(\frac{|v'-v_0|}{\eps} \right)  \\
	& \qquad  \psi \left(\frac{|x-(x_0+\tau_+(x_0,v_0)v_0|)}{\eps_1} \right)\psi \left(\frac{|v-v_0|}{\eps_1} \right)\rd{\xi(x,v)}\rd{\xi(x',v')}\\	
	\le & C\sup_t\int_{\Gamma_-}\int_{V}  \frac{|n(x')\cdot v'||n(x'+tv'+\tau_+(x'+tv',v)v) \cdot v|}{n(x'+tv'+\tau_+(x',+tv',v)v)\cdot v}     \psi \left(\frac{|x'-x_0|}{\eps} \right)\psi \left(\frac{|v'-v_0|}{\eps} \right)   \\
	& \qquad \psi \left(\frac{|(x'+tv'+\tau_+(x'+tv',v)v)-(x_0+\tau_+(x_0,v_0)v_0)|)}{\eps_1} \right)  \psi \left(\frac{|v-v_0|}{\eps_1} \right) \rd{v}\rd{\xi(x',v')}\\
	\le & C\int_{|x'-x_0|<\eps,|v'-v_0|<\eps,|v-v_0|<\eps_1}\rd{v}\rd{\xi(x',v')} \le C\eps^4 \eps_1^2 \,,
	\end{split}
	\end{equation*}
	where in the first inequality we bound the exponential terms and the $k$ term by $C$ and then integrate out the $x$ variable.
	
	For $E_3$,  we have  
	\begin{equation*}
	\begin{split}
	|E_3| \le &  \int_{|x'-x_0|<\eps,|x-(x_0+\tau_+(x_0,v_0)v_0|<\eps_1,|v'-v_0|<\eps,|v-v_0|<\eps_1}\alpha_3(x,v; x',v')\rd{\xi(x,v)}\rd{\xi(x',v')} \\
	\le & \int_{|x'-x_0|<\eps,|v'-v_0|<\eps}\|\alpha_3(\cdot,\cdot;x',v')\|_{L^p(\rd{\xi})}\|{\bf 1}_{|x-(x_0+\tau_+(x_0,v_0)v_0|<\eps_1,|v-v_0|<\eps_1}\|_{L^{p'}(\rd{\xi})}\rd{\xi(x',v')} \\
	\le &  C_p\eps^{4/p'}\int_{|x'-x_0|<\eps,|v'-v_0|<\eps}\rd{\xi(x',v')} \\
	\le & C_p\eps_1^{4/p'} \eps^4\,,
	\end{split}
	\end{equation*}
	thanks to Proposition \ref{thm1} and H\"older inequality. Notice that $4/p'=2-\delta$ if $p=\frac{2}{1+\delta/2}<2$, and \eqref{eqn:E123} directly follows. To go from \eqref{eqn:E123} to \eqref{eqn:phi123}, one just needs to notice that $\phi_1 = \lim_{\eps \rightarrow 0, \eps_1 \rightarrow 0} E_1$ and $\phi_{R,1} = \lim_{\eps \rightarrow 0, \eps_1 \rightarrow 0} E$.
	\end{proof}

\subsection{Study of $\Sigma-\sigma$}
We study the error in the final recovery. Comparing~\eqref{eqn:opt-cont1} and~\eqref{eqn:opt-dis1-regularizing}, we see that the true media $\sigma$ minimizes:
\begin{equation*}
\min_{\sigma} \mathcal{F}\left(\mathcal{R}[\sigma](x,v) - {a}(x,v) \right)\,,
\end{equation*}
or directly:
\begin{equation*}
\mathcal{R}[\sigma](x,v) = {a}(x,v)\,,
\end{equation*}
while the numerical recovery $\Sigma$ satisfies:
\begin{equation*}
\min_{\Sigma} \|  \Rmat \cdot \Sigma - \vec{a} \|_X + \lambda\| \Sigma\|_Y\,,
\end{equation*}
with $a$ and $\vec{a}$ defined in~\eqref{eqn:def_cont_a} and~\eqref{eqn:def_dis_a}.

The difference between $\mathcal{R}[\sigma]$ and $\Rmat\cdot\Sigma$ is governed by the accuracy of the quadrature rule. Suppose the second order trapezoidal rule is used to approximate the line integral of $\mathcal{R}$, and then the truncation error is given by, for each experiment:
\begin{equation}\label{eqn:trapezoidal}
\mathcal{R}[\sigma](x^{(k)},v^{(k)}) - (\Rmat\cdot\sigma^\text{dis}_a)_k = \mathcal{O}(\Delta x^2)\,.
\end{equation}

The difference between $a$ and $\vec{a}$, according to the definition, is from the error in $\phi_1$:
\begin{lemma}\label{lemma:a}
With incoming data given by $f_-(x',v') = \psi \left(\frac{|x'-x^{(j)}|}{\eps} \right)\psi \left(\frac{|v'-v^{(j)}|}{\eps} \right)$, the analytical $a(x,v)$ and the discrete $\vec{a}$ differ by $\mathcal{O}(\eps_1^{-2-\delta}\eps^4)$ (with arbitrary small $\delta>0$), i.e., 
\begin{equation*}
a(x^{(j)}_*, v^{(j)}_*)-\vec{a}_j= \ln \left(  \frac{f_-(x^{(j)}, v^{(j)})}{\phi_1(x^{(j)}_*, v^{(j)}_*)}   \right)- \ln \left( \frac{\vec{f}_-(x^{(j)}, v^{(j)})}{\phi_{R,1}(x^{(j)}_*, v^{(j)}_*)}   \right) = \mathcal{O}(\eps_1^{-2-\delta}\eps^4)\,,
\end{equation*}
where $(x^{(j)}, v^{(j)}) \in \Gamma_-$, and $(x^{(j)}_*, v^{(j)}_*) \in \Gamma_+$ is defined in \eqref{xv_star}, both of them are on the grid points. 
\end{lemma} 
\begin{proof}
It is a immediate consequence of Theorem~\ref{thm:separation} and the definition of $a$.
\end{proof}

We then have the following theorem.
\begin{theorem} \label{error-sigmaa}
	If the norms in \eqref{eqn:opt-dis1-regularizing} are both taken as the $L^2$ norm and assume the range condition: there exists a vector $z$ such that
	\begin{equation}\label{eqn:range_condition}
	\sigma^{dis} = \Rmat^T z \,.  
	\end{equation}
	Then by choosing
	\begin{equation}\label{eqn:lambda_choice}
	\lambda = \frac{\eps_1^{-2-\delta} \eps^4+\Delta x^2}{\|z\|_2}
	\end{equation}
	in \eqref{eqn:opt-dis1-regularizing} with $\delta>0$, one has the error estimate
	\begin{equation}\label{eqn:sigma_error}
	\|\Sigma-\sigma^{dis}\|_2 \le C\|z\|_2^{1/2}(\eps_1^{-2-\delta} \eps^4 + \Delta x^2)^{1/2}  \,.
	\end{equation}
\end{theorem}
\begin{proof}
The major part of the proof follows a standard result from Tikhonov regularization, as summarized in \cite{Book_vogel,Engl_regularization_book}. First according to Lemma~\ref{lemma:a} and equation~\eqref{eqn:trapezoidal}, one has
\begin{equation*}
\vec{a} = a(x,v) +\mathcal{O}(\eps_1^{-2-\delta} \eps^4) = \Rmat\cdot\sigma^{dis} +\mathcal{O}(\eps_1^{-2-\delta} \eps^4+\Delta x^2)\,,
\end{equation*}
and thus 
\begin{equation} \label{sigma-dis}
\Rmat \cdot \sigma^{dis} = \vec{a} + C \left(\eps_1^{-2-\delta} \eps^4 + \Delta x^2 \right)\,.
\end{equation}
If considering the $L^2$ norm, i.e., \eqref{eqn:opt-dis1-regularizing} writes 
\begin{equation*}
\min_{\Sigma} \|  \Rmat \cdot \Sigma - \vec{a} \|_2^2 + \lambda\| \Sigma\|_2^2\,,
\end{equation*}
then the minimizer $\Sigma$ reads
\begin{equation}\label{Sigma}
\Sigma = (\Rmat^T \Rmat + \lambda \Imat)^{-1} \Rmat^T \vec{a}\,.
\end{equation}
Comparing \eqref{sigma-dis} and \eqref{Sigma}, their error can be computed as 
\begin{eqnarray*}
\Sigma - \sigma^{dis} &=& (\Rmat^T \Rmat + \lambda \Imat)^{-1} \Rmat^T \Rmat \sigma^{dis} - \sigma^{dis}  - (\Rmat^T \Rmat + \lambda \Imat)^{-1} \Rmat^T C(\eps_1^{-2-\delta} \eps^4 + \Delta x^2)
\\ &=& \left[ (\Rmat^T \Rmat + \lambda \Imat)^{-1} \Rmat^T \Rmat \sigma^{dis} - \sigma^{dis}\right] - (\Rmat^T \Rmat + \lambda \Imat)^{-1} \Rmat^T C (\eps_1^{-2-\delta} \eps^4 + \Delta x^2)
\\ &=& e^{reg} + e^{qua}\,,
\end{eqnarray*}
where the first part $e^{reg}$ is the regularization error, and the second $e^{qua}$ is the error from computing $\vec{a}$ and may get amplified in the optimization process. 

Now write the singular value decomposition of $\Rmat = U \Lambda V^T$, and denote the columns of $U$ and $V$ as $u_i$ and $v_i$ respectively, and the elements in $\Lambda$ is $s_i$. Then we have
\begin{eqnarray} \label{e-qua}
\|e^{qua}\|_2 &=& \left\| \sum_i v_i (s_i^2 + \lambda)^{-1} s_i u_i^T C (\eps_1^{-2-\delta} \eps^4+ \Delta x^2)  \right\|_2\nonumber
\\ &\leq & \frac{C}{\sqrt{\lambda}} (\eps_1^{-2-\delta} \eps^4 + \Delta x^2)\,,
\end{eqnarray}
where we have used Cauchy-Schwarz inequality and the fact that $s_i/(s_i^2+\lambda) \leq \frac{1}{\sqrt{\lambda}}$ to get the inequality. For $e^{reg}$, using the range condition \eqref{eqn:range_condition}, we have
\[
e^{reg} = \sum_i v_i \frac{-\lambda}{s_i^2 + \lambda} s_i (u_i ^T z)\,,
\]
and therefore,
\begin{eqnarray} \label{e-reg}
\|e^{reg}\| = \sum_i \left(\frac{\lambda s_i}{\lambda + s_i^2} \right) (u_i^T z)^2 \leq \|z\|_2^2 \max_i \left( \frac{\lambda s_i}{\lambda + s_i^2} \right)^2  \leq C \lambda \|z\|_2^2\,.
\end{eqnarray}
Combining \eqref{e-reg} and \eqref{e-qua}, we have 
\begin{equation*}
\| \Sigma - \sigma^{dis} \|_2 \leq C \left( \sqrt{\lambda} \|z\|_2 + \frac{1}{\sqrt{\lambda}} (\eps_1^{-2-\delta} \eps^4 + \Delta x^2) \right)\,,
\end{equation*}
then choosing $\lambda$ from \eqref{eqn:lambda_choice}, the result \eqref{eqn:sigma_error} directly follows. 
 \end{proof}

\section{Discussion in diffusive regime}
As demonstrated in previous sections, in most optimization formulation of the inverse problem, one always needs a repeated use of forward solver. However, the radiative transfer equation resides in a high dimensional phase space, which requires a large amount of computation effort. A well accepted approximation is the diffusion approximation, which gives rise to a model that only varies in spatial domain. This approximation turns out to be very efficient in the forward setting, but brings huge error in the inverse problem. Studies have shown that, in the case when such approximation can be made, the recovery of the scattering and absorption coefficient becomes unstable and inaccurate. This phenomena was systematically studied in~\cite{ChenLiWang} for the stationary case, where the Knudsen number ($\Kn$) denotes the regime of the equation: smaller Knudsen number means better approximation of the diffusion limit. Then it is shown that, as the Knudsen number shrinks to zero, in the forward setting, the RTE equation converges to the diffusion equation, with its scattering and absorption coefficients becoming the diffusion and the damping coefficients in the diffusion equation. For the inverse setting, however, the recovery becomes very bad, with the so-defined indistinguishability coefficient---a quantity that measures the accuracy of the recovery, blowing up to infinity in the diffusion regime.

In this section, we will revisit this result in our numerical optimization framework and show that, in 3D diffusive regime, $E_1$ cannot be separated from $E_2$ and $E_3$, making the algorithm invalid in the very first step. More precisely, the original equation \eqref{eqn:000} in the diffusive scaling rewrites as 
\begin{equation} \label{eqn:001_kn}
v\cdot\nabla_xf +\left( \Kn \sigma_a(x) + \frac{1}{\Kn} \sigma_\nu(x,v) \right) f =  \frac{1}{\Kn} \int k(x,v,v')f(x,v')\rd{v'} \,,
\end{equation}
where the Knudsen number $\Kn$ represents the ratio of the mean free path and the domain length, and is an indicator of the regime the equation is in. Here we decompose the the total absorption $\sigma$ in \eqref{eqn:000} into two components, $\sigma_a$ and $\sigma_\nu$. The former one is a pure absorption, and the latter one is the absorption caused by scattering, i.e., $\sigma_\nu(x,v) \int k(x,v,v') dv'$. As written in \eqref{eqn:001_kn}, $\sigma_a$ and $\sigma_\nu$ is rescaled differently. In this new form, the inverse problem will recover either $(\sigma, k)$ or $(\sigma_a, k)$.

For simplicity, we consider the domain $\Omega = \left\{ x=(x_1,x_2, x_3) \in \mathbb{R}^3, 0 \leq x_1 \leq 1\right\}$, and thus $\partial \Omega$ are two infinite size parallel walls located at $x_l=(0,0,0)$ and $x_r=(1,0,0)$ respectively. The outer normal directions are then $n_l = (1,0,0)$ and $n_r = (-1,0,0)$. 
The main theorem states as follows. 
\begin{theorem} \label{error_Kn}
In 3D, consider the incoming data given by $f_- (x',v') = \psi\left( \frac{|x'-x_0|}{\eps}\right) \psi \left( \frac{|v'-v_0|}{\eps}\right)$ with 
$x_0= (0,x_2,x_3)$, and $E_i$ defined in \eqref{eqn:E_i}. Assume that $\partial\Omega$ is $C^1$ in the neighborhood of $x_* = x_0+\tau_+(x_0,v_0)v_0$, one has:
\begin{equation} \label{E1}
E_1 \le C_1\eps^4\exp{\left(-C_2/\Kn\right)}\,,\quad E_3 \ge  c\eps^4 \eps_1^4 \Kn^q\,,  
\end{equation}
where $C_1$, $C_2$, $c$ and $q$ are positive constants, and $q$ depends on the dimension of the problem. Consequently, if
\begin{equation*}
\Kn\leq \mathcal{O}\left(-\frac{1}{\ln\eps_1} \right)\,,
\end{equation*}
the algorithm breaks down since $E_3$ can no longer be distinguished from $E_1$.
\end{theorem}

The proof of this theorem relies on asymptotic and boundary layer analysis. For this reason, we first consider a 3D case with slab geometry, which essentially reduces to a problem in 1D. Indeed, denote $x=(x_1,x_2,x_3)$, then in slab geometry, the dependent functions are assumed to be homogeneous in $x_2$ and $x_3$, and the velocity is $v =(\cos\theta,0,0)$, where $\theta$ is the angle between the direction of the flight and positive $x_1$ direction. Therefore, the photon dynamics varies only along $x_1$ and $v_1$, and the problem \eqref{eqn:001_kn} reduces to a 1D problem:
\begin{equation}\label{eqn:RTE_Kn}
\begin{cases}
v_1\partial_{x_1} f = \frac{1}{\Kn}\left(\langle f\rangle - f\right)\,,\quad(x_1,v_1)\in[0,1]\times[-1,1]\\
f|_{x_1=0,v_1>0} = \phi(v_1) \,,\\
f|_{x_1=1,v_1<0} = 0 \,.
\end{cases}
\end{equation}
Here we assume that the boundary is placed at $x_1=0$ and $x_1=1$. At the left boundary $x_1=1$, there is an incoming data $\phi(v_1)$, which could be designed as a concentrated source term. We also assume that $\sigma_a \equiv 0$ just for the ease of computation that follows.  

When $\Kn = 1$, equation \eqref{eqn:RTE_Kn} is a reduced version of RTE~\eqref{eqn:000} in 1D with $\sigma= k = 1$. And as $\Kn\to 0$, physically it means the interactions between particles become intense, driving the equation to the diffusive regime. More specifically, we have (here we omitted the subscript `$1$'):
\begin{proposition}
In the $\Kn\to 0$ limit, the solution to~\eqref{eqn:RTE_Kn} can be well approximated by:
\begin{equation}\label{eqn:f_f_A_E_A}
f (x,v)= f^A + E^A = f^L\left(\frac{x}{\Kn},v\right) + f^I(x,v) + E^A(x,v) \,,
\end{equation}
where $f^A$ is the approximate solution and has two parts: the layer part, denoted by $f^L$ and the interior part, denoted by $f^I$. $E^A$ is the approximation error. There exists a constant $\eta$, such that the layer, the interior and the error satisfy the following equations respectively:
\begin{itemize}
\item[$f^L$:] The layer lives only within $\Kn$ distance from $x=1$ and with change of variables $y=\frac{x}{\Kn}$, it satisfies:
\begin{equation}\label{eqn:f_L}
\begin{cases}
v\partial_y f^L = \langle f^L\rangle - f^L \,, \\
f^L(y=0,v>0) = \phi(v) - \eta \,, \\
f^L(y=\infty) = 0\,;
\end{cases}
\end{equation}
\item[$f^I$:] The interior is defined as:
\begin{equation*}
f^I(x,v) = \theta - \Kn v\partial_{x}\theta \,,
\end{equation*}
with $\theta$ satisfies:
\begin{equation}\label{eqn:diff_limit}
\begin{cases}
\partial^2_{x}\theta =  0\,,\\
\theta(x=0) = \eta\,,\quad\theta(x=1) = 0\,.
\end{cases}
\end{equation}
It also means that $f^I$ satisfies:
\begin{equation}\label{eqn:f_I}
\begin{cases}
v\partial_{x} f^I = \frac{1}{\Kn}\left(\langle f^I\rangle - f^I\right)  \,,\\
f^I(x=0,v>0) = \eta - \Kn v\partial_{x}\theta \,, \\
f^I(x=1,v<0) = 0\,;
\end{cases}
\end{equation}
\item[$E^A$:] The error term satisfies:
\begin{equation}\label{eqn:E_A}
\begin{cases}
v\partial_{x} E^A = \frac{1}{\Kn}\left(\langle E^A\rangle - E^A\right)\\
E^A(x=0,v>0) = \Kn v\partial_{x}\theta\\
E^A(x=1,v<0) =- f^L(\frac{1}{\Kn},v)
\end{cases}\,.
\end{equation}
\end{itemize}
Moreover, in a neighborhood of the right wall, with $x\in(1-2\Kn,1)$:
\begin{equation}\label{eqn:f_properties}
f^I(x) = 1-x\,,\quad f^L(x) \sim e^{-x/\Kn}\,,\quad E^A\sim e^{-x/\Kn}\,.
\end{equation}
\end{proposition}
\begin{proof}
To show~\eqref{eqn:f_f_A_E_A}, one simply needs to add up the three equations~\eqref{eqn:f_L},~\eqref{eqn:f_I} and~\eqref{eqn:E_A}. The properties in~\eqref{eqn:f_properties} are from: the solution to the diffusion equation and the behavior of the layer equation~\cite{LiLuSun}. $\eta$ is termed the extrapolation length, and its existence is proved in \cite{LiLuSun}, but it can not be computed explicitly.
\end{proof}

\begin{remark}
In more general 3D case, the analysis can be a bit complicated. As mentioned earlier in this section, we still assume that the boundaries are two infinite size parallel walls located at $x_l=(0,0,0)$ and $x_l=(1,0,0)$ respectively. Then the equation writes:
\begin{equation*}
\begin{cases}
v_1\partial_{x_1}f+v_2\partial_{x_2}f+v_3\partial_{x_3}f = \frac{1}{\Kn}\mathcal{L}[f]\,,\quad (x_1,x_2,x_3)\in[0,1]\times\mathbb{R}^2\,,\quad\sum_{i=1}^3v_i^2 = 1\,,\\
f|_{x_1=0,v_1>0} = \phi^x(x_2,x_3)\phi^v(v)\,,\\
f|_{x_1=1,v_1<0} = 0\,.
\end{cases}
\end{equation*}
Here $\mathcal{L}[f] = \average{f} -f $ is an abbreviation of the collision term, and $\phi^x$ and $\phi^v$ are two functions concentrated at $x_0=(0,0,0)$ and $v_0=(1,0,0)$. Namely, the particles getting into the domain are mainly from the origin with speed pointing directly to the wall on the right. As $\Kn$ goes to zero, the layer appears only on the left side along $x_1$ direction, and by setting:
\begin{equation*}
y=\frac{x_1}{\Kn} \,,
\end{equation*}
we separate the layer equation and the interior equation:
\begin{itemize}
\item[$\bullet$]{Layer:}
\begin{equation}\label{eqn:layer_3D}
\begin{cases}
v_1\partial_y f^L = \mathcal{L}[f^L]\,,\quad y\in[0,\infty) \,,\\
f^L|_{y=0,v_1>0} = \phi^x(x_2,x_3)\left(\phi^v(v)-\eta\right)\,, \\
f^L|_{y=\infty} = 0\,;
\end{cases}
\end{equation}
\item[$\bullet$]{Interior:}
\begin{equation}\label{eqn:interior_3D}
\begin{cases}
v_1\partial_{x_1}f^I+v_2\partial_{x_2}f^I+v_3\partial_{x_3}f^I = \frac{1}{\Kn}\mathcal{L}[f^I]\,,\quad (x_1,x_2,x_3)\in[0,1]\times\mathbb{R}^2   \,,\\
f^I|_{x_1=0,v_1>0} = \eta\phi^x(x_2,x_3) \,,\\
f|_{x_1=1,v_1<0} = 0\,;
\end{cases}
\end{equation}
\item[$\bullet$]{Error:} the error is defined by:
\begin{equation*}
E^A = f-f^A = f-f^L-f^I\,,
\end{equation*}
and taking~\eqref{eqn:layer_3D} and~\eqref{eqn:interior_3D} into account, one gets:
\begin{equation}\label{eqn:error_3D}
\begin{cases}
\left(v_1\partial_{x_1}+v_2\partial_{x_2}+v_3\partial_{x_3}\right)E^A = \frac{1}{\Kn}\mathcal{L}[E^A] - (v_2\partial_{x_2}+v_3\partial_{x_3})f^L\,,\quad (x_1,x_2,x_3)\in[0,1]\times\mathbb{R}^2 \,, \\
E^A|_{x_1=0,v_1>0} = 0 \,, \\
E^A|_{x_1=1,v_1<0} = -f^L(\frac{1}{\Kn},x_2,x_3,v)  \,.
\end{cases}
\end{equation}
\end{itemize}

It has been shown in~\cite{LiLuSun} that $f^L(y,v)\in L^2\left(e^{\beta y}\rd{y},L^2(\rd x)\right)$, then in the neighborhood of $x_{0*}=(1,0,0)$, $v_{0*}=(1,0,0)$,
\begin{equation*}
f^L(x,v) \sim e^{-\frac{x_1}{\Kn}}\,,\quad E^A(x,v) \sim e^{-\frac{x_1}{\Kn}}\,,
\end{equation*}
and thus $f\sim f^I$ around $x_1=1$. The standard asymptotic analysis applied on $f^I$ shows that:
\begin{equation*}
f^I = \theta -\Kn v\cdot\nabla_x\theta +\Kn^2 (v\cdot\nabla_x)^2\theta +\cdots\,, \qquad \int f^I\rd{v} = \theta\,,
\end{equation*}
where
\begin{equation*}
\Delta_{x} \theta = 0\,,\quad \theta|_{x_1=0} = \eta\phi^x(x_2,x_3)\,,\quad \theta|_{x_1=1} = 0\,, \qquad x_\ast = (1,0,0)\,.
\end{equation*}
In summary, around $x_{0*}$, asymptotically we have $$\int f\rd{v}\sim \int f^I\rd{v} = \theta\sim |x-x_{0\ast}|^q$$ with $q$ depending on the dimension ($q=1$ in 1D for example). 

At the end of this remark, we would like to point out that the choice of $x_0$ and $v_0$ is arbitrary as long as they resides on the left wall. Here we pick $x_0=(0,0,0)$ and $v_0=(1,0,0)$ just as an example. In the following proof of Theorem~\ref{error_Kn}, we still make this choice and again it can be easily adapted to other choices. 
\end{remark}

\begin{proof}[Proof of Theorem~\ref{error_Kn}]
	It follows from~\eqref{eqn:error_1} that 
	\begin{equation*}
	E_1 \le e^{-C_2/\Kn}\int_{|x'-x_0|<\eps,|v'-v_0|<\eps}\rd{\xi}(x',v')\,,
	\end{equation*}
	where $C_2=\min_{x',v'}{\tau_-((x'+\tau_+(x',v')v'),v')}>0$, the minimum taken over all possible $(x',v')$ for which the integrand in~\eqref{eqn:error_1} is nonzero. Then the upper bound for $E_1$ \eqref{E1} follows. 
	
	In order to estimate $E_3$, first notice that for the input $\psi \left(\frac{|x'-x_0|}{\eps} \right)  \psi \left(\frac{|v'-v_0|}{\eps} \right)$ with $x_0 = (0,0,0)$ and $v_0 = (1,0,0)$, the asymptotic analysis gives 
	\begin{equation*}
	f(x,v) = \theta(x) + \mathcal{O}(e^{-1/\Kn}),\quad   \theta(x) \geq c|x-x_\ast|^q
	\end{equation*}
	for some $c>0$, where $x_{0\ast} =(1,0,0)$. This means at the receiver placed at $x_{0\ast}$, $\theta(x)$ has a polynomial lower bound. Then, for $x$ with $|x-x_{0\ast}|\le \Kn$, by using~\eqref{eqn:716}, we estimate
	\begin{equation*}
	\begin{split}
	(\opB \opA^{-1}\opB f)(x,v) = & \int_{\Omega}\int_{V}K_1(x,v,y,w)f(y,w)\rd{w}\rd{y} \\ 
	= & \Kn^{-2}\int_{\Omega_x}e^{-|x-y|/\Kn}\frac{\theta(y)}{|x-y|^2} \rd{y} \\
	\ge & c\Kn^{-2}\int_{\Omega_x\bigcap \{|x-y|\le \Kn\}} \frac{|y-x_\ast|^q}{|x-y|^2} \rd{y} \\
	\ge & c\Kn^{-2}\int_{\Omega_x\bigcap \{|x-y|\le |x-x_\ast|/2\}} \frac{|y-x_\ast|^q}{|x-y|^2} \rd{y} \,.\\
	\end{split}
	\end{equation*}
	Note that if $|x-y|\le |x-x_\ast|/2$, then $|y-x_\ast|\ge |x-x_\ast|/2$. Thus
	\begin{equation*}
	(\opB \opA^{-1}\opB f)(x,v) 
	\ge c\Kn^{-2}|x-x_\ast|^q\int_{\Omega_x\bigcap \{|x-y|\le |x-x_\ast|/2\}} \frac{1}{|x-y|^2} \rd{y} \ge c\Kn^{-2}|x-x_\ast|^{q+1}\,.
	\end{equation*}
Since $\alpha_3$ relates to $\mathcal{A}^{-1}\left(\opB\opA^{-1}\opB\right)$, and using the assumption that $\partial\Omega$ is $C^1$ at $x_{0\ast}$, we have
	\begin{equation*}
	\begin{split}
	\alpha_3(x_\ast,v;x',v') = & \int_0^{\tau_-(x_\ast,v)}e^{-\frac{1}{\Kn}\int_0^t\sigma_\nu(x_\ast-s v,v)\rd{s}}(\opB \opA^{-1}\opB f)(x_\ast-t v,v)\rd{t}  \\
	&  \ge c\Kn^{-2}\int_0^{\Kn} |(x_\ast-t v)-x_\ast|^{q+1} \rd{t} = c\Kn^q\,.\\
	 \end{split}
	\end{equation*}
	Then the lower bound of $E_3$ follows directly from its definition~\eqref{eqn:E_i}. Comparing $E_3$ and $E_1$, we see that as long as $\Kn\leq\mathcal{O}(-\frac{1}{\ln\eps_1})$, $E_3$ is no longer much smaller than $E_1$, and the separation cannot be done.
	
\end{proof}


\bibliographystyle{siam}
\bibliography{inverse_numerics}

\end{document}